\documentclass[a4paper,11pt,reqno]{amsart}

\usepackage[utf8]{inputenc}
\usepackage[T1]{fontenc}

\usepackage{paralist}				
\usepackage{letterspace}
\usepackage{ifdraft}
\usepackage{cite}
\usepackage{microtype}
\usepackage{tikz,tikz-cd}
\usepackage[hidelinks]{hyperref}

\usepackage{amsmath}
\usepackage{amssymb}
\usepackage{amsthm}				
\usepackage{mathtools}				
\usepackage{mathrsfs}				
\usepackage{dsfont}				

\newcommand{\mf}{\mathfrak}

\newcommand{\msc}{\mathscr}

\newcommand{\IZ}{\mathbb{Z}}
\newcommand{\IQ}{\mathbb{Q}}

\newcommand{\IC}{\mathbb{C}}

\newcommand{\ga}{\alpha}
\newcommand{\gb}{\beta}

\newcommand{\bP}{\mathbf{P}}

\newcommand{\ssq}{\subseteq}

\newcommand{\Ber}{\mathsf{Ber}}

\newcommand{\Di}{\,\textnormal{d}}

\newcommand{\bdry}{\partial}
\newcommand{\rt}{\textnormal{root}}

\newtheorem{lemma}{Lemma}

\newtheorem{prop}[lemma]{Proposition}
\newtheorem{thm}[lemma]{Theorem}
\newtheorem{conj}[lemma]{Conjecture}
\theoremstyle{remark}
\newtheorem{rem}[lemma]{Remark}

\theoremstyle{definition}
\newtheorem{defn}[lemma]{Definition}

\title[Polynomial invariants for rooted trees]{Polynomial invariants for rooted trees related to their random destruction}
\author{Fabian Burghart}
\address{Department of Mathematics and Computer Science, Eindhoven University of Technology, 5612AE Eindhoven, The Netherlands}
\email{f.burghart@tue.nl}

\date{05 October 2024}

\begin{document}

\begin{abstract}
 We consider three bivariate polynomial invariants $P$, $A$, and $S$ for rooted trees, as well as a trivariate polynomial invariant $M$. These invariants are motivated by random destruction processes such as the random cutting model or site percolation on rooted trees. We exhibit recursion formulas for the invariants and identities relating $P$, $S$, and $M$. The main result states that the invariants $P$ and $S$ are complete, that is they distinguish rooted trees (in fact, even rooted forests) up to isomorphism. The proof method relies on the obtained recursion formulas and on irreducibility of the polynomials in suitable unique factorization domains. For $A$, we provide counterexamples showing that it is not complete, although that question remains open for the trivariate invariant $M$. 
\end{abstract}

\maketitle

\section{Introduction and preliminaries}

The study of polynomial invariants in graph theory is of considerable tradition, with perhaps the best-known invariant being the Tutte polynomial \cite{Tutte47,Tutte54}. For trees on $n$ vertices, it is well-known that the Tutte polynomial evaluates to $x^{n-1}$ and is thus of little use when investigating trees. To overcome this issue, Chaudhary, Gordon and McMahon in \cite{GM89} and \cite{CG91Tutte} defined specific Tutte polynomials for (rooted) trees by replacing the usual rank of a subgraph in the corank-nullity definition of the Tutte polynomial by different notions of tree rank. In these papers, several of the obtained (modified) Tutte polynomials introduced for rooted trees were shown to be \emph{complete} invariants, that is, no two non-isomorphic rooted trees are assigned the same polynomial. 

Since then, more complete polynomial invariants for rooted trees were found, such as polychromatic polynomials \cite{BR2000} and the rooted multivariable chromatic polynomial \cite{LW22} -- both invariants require a large number of variables. The bivariate Ising polynomial \cite{AK09} and the Negami polynomial \cite{NO96}, originally defined for unrooted trees, were later shown to have versions for rooted trees that are complete invariants, see \cite{law2011}. More recently, Liu \cite{Liu21tree} found a complete bivariate polynomial as a generating function for a certain class of subtrees, and \cite{RW22polynomial} considers an extension of Liu's polynomial to three variables. 

In this paper, we define several polynomial invariants for rooted trees that are defined combinatorially, but can be motivated by two models for the random destruction of trees, namely Bernoulli site percolation and the random cutting model. Among these polynomials, two bivariate invariants are proven to be complete using an approach via irreducibility of polynomials and a suitable recursion, and for two more invariants examples are provided showing that they are not complete. These results suggest in a non-rigorous way that complete knowledge about the behaviour of a tree under random destruction should uniquely determine the tree, but it is still open if this holds rigorously (see e.g. the discussion below Conjecture~\ref{conj}). However, all polynomial invariants considered here are closely related, leading to several identities that might be interesting in their own right, or for the purpose of explicit computations relating to phenomena around the random cutting model or site percolation, like the recursions in Lemmas~\ref{lemma:polrec},~\ref{lemma:SArec},~\ref{lemma:Mrec}.

\subsection*{Structure of the paper}
After fixing the necessary notation and terminology concerning trees below, Section~\ref{sec:defs} is dedicated to the combinatorial definitions of our polynomial invariants. Section~\ref{sec:prob} delivers the probabilistic background on random destruction of trees, and may serve as a motivation for the polynomial invariants, but the material presented there is not necessary for the main results or the proofs thereof in earlier sections. Accordingly, a reader not interested in the relation between random tree destruction and the polynomial invariants may safely skip this section. Section~\ref{sec:ident} returns to the combinatorial setting, and features several technical results like recursion identities for all polynomials. In Section~\ref{sec:leaf-induced} we formulate and prove the main theorem of the paper, Theorem~\ref{thm:complete}, and employ it to derive a reconstruction result for leaf-induced subtrees. Finally, Section~\ref{sec:discussion} contains several remarks, examples, and an open conjecture.

\subsection*{Preliminaries}
For the purpose of this paper, a \emph{rooted tree} $T$ is a finite tree with one distinguished vertex, called the root of $T$. It will be convenient to also consider \emph{rooted forests}, by which we understand a finite (but possibly empty) disjoint union of rooted trees. By this convention, every component in a rooted forest is a rooted tree. A vertex is a \emph{leaf} of a rooted forest if it does not have any children (thus an isolated vertex is simultaneously a root and a leaf). 

An isomorphism of rooted forests is a graph isomorphism that additionally maps roots to roots. 

Given a rooted tree $T$, denote by $r$ the number of children of the root node $v_0$. We can construct a rooted forest from $T$ by removing $v_0$, thus creating a forest with $r$ components, and declaring the unique child of $v_0$ in each component to be the root node in that component. We will denote the resulting forest by $T-v_0$. The components of $T-v_0$ are also called the \emph{branches} of $T$. 

Conversely, given a rooted forest $F$ with $r\geq0$ components, let $v_0$ be a vertex not in $F$ and draw an edge from $v_0$ to each of the $r$ roots in $F$. Upon declaring $v_0$ to be the root of the so-constructed tree, we have obtained a rooted tree. We will denote the resulting tree by $\wedge(F)$ or $\wedge(T_1,\dots,T_r)$ if $F$ is given by its components $T_1,\dots,T_r$. 

Since our definition allows for empty rooted forests (containing no vertices whatsoever), it follows immediately that $\wedge(F)-v_0\cong F$ and $\wedge(T-v_0)\cong T$ for all rooted forests $F$ and all rooted trees $T$. In particular, removing the root of a tree and adding a joint root to a forest are inverse bijections between isomorphism classes of rooted trees and isomorphism classes of rooted forests. 

For convenience, $\bullet$ will denote the rooted tree on one vertex.

\section{Setting the stage: Defining polynomials}\label{sec:defs}

\subsection*{Leaf-induced subforests} 
Let $F$ be a rooted forest. By a \emph{leaf-induced subforest} $F'$ we understand a rooted forest $F'$ that is a (possibly empty) union of paths connecting roots of $F$ to leaves of $F$. In other words, any leaf of $F'$ must also be a leaf of $F$. It follows that $F'$ is completely determined by choosing a subset of the leaves of $F$, and connecting each of the chosen leaves to the root of its component. In particular, if $F$ has $\ell$ leaves, then it has $2^\ell$ leaf-induced subforests. 

\begin{defn}\label{defn:pol}
 For a rooted forest $F$, denote by $P_F(x,y)$ the bivariate generating function for leaf-induced subforests of $F$ according to their number of vertices and leaves. That is,
 \begin{equation}\label{eq:Pdef}
  P_F(x,y) = \sum_{F'\ssq F\ \text{leaf-induced}} x^{|V(F')|} y^{|L(F')|},
 \end{equation}
 where $V(F')$ and $L(F')$ denote the sets of vertices and leaves of $F'$, respectively. 
\end{defn}

As an example, if $T$ is the path on $n$ vertices with the root situated on one end, then $P_T(x,y)=1+x^ny$, since the only leaf-induced subforests are the empty one and $T$ itself. For $T$ being the star on $n+1$ vertices, with the root being the central vertex, we have $P_T(x,y)=1+\sum_{k=1}^n \binom{n}{k} x^{k+1}y^k$, which can be seen directly from a combinatorial argument, or computed recursively as will be established in the next section. 

For a rooted tree $T$, it will also be useful to introduce the shorthand notation 
\begin{equation}\label{eq:pdef}
 p_T(q):=1-P_T(q,-1)
\end{equation}
for the univariate generating function of non-empty leaf-induced subtrees with a sign according to the parity of the number of leaves. 

It should be mentioned that Razanajatovo Misanantenaina and Wagner, in \cite{RW22polynomial}, considered a trivariate polynomial invariant $\mathcal{P}_T(x,y,z)$ defined recursively by $\mathcal{P}_\bullet(x,y,z)=x$ and 
\[
 \mathcal{P}_T(x,y,z) = yz^{|T|-1} + \prod_{i=1}^r \mathcal{P}_{T_i}(x,y,z)
\]
for a tree $T=\wedge(T_1,\dots,T_r)$. Their Propositions 2.16 and 2.17 and the comment thereafter establish a connection between $\mathcal{P}_T$ and $P_T$, given by 
\[
 P_T(x,y)=x^{|T|}\mathcal{P}\left(y+\frac{1}{x},\frac{1}{x}-1,\frac{1}{x}\right).
\]

We also mention that $p_T(q)$ was previously investigated in \cite{devroye2011,2013transversals} in the context of transversals in trees, where a is a set of vertices intersecting all paths from the root to the leaves.

\subsection*{Admissible subtrees} By a \emph{subtree} of a rooted tree $T$ we mean either the empty subgraph of $T$ or any connected subgraph of $T$ that contains the root (though we will break with this convention in the context of fringe subtrees which generally do not contain the root node of $T$, see the paragraph above Definition~\ref{defn:polM}). Since a subtree $T'$ of $T$ is uniquely determined by its vertex set, we will not distinguish between $T'$ and its vertex set.  

We say that a subtree $T'$ is \emph{admissible} if and only if it is empty, or if $T'$ contains the root of $T$ but none of the leaves of $T$. We write $\msc A(T)$ for the set of all admissible subtrees of $T$. 

Given a set $S$ of vertices in a rooted tree $T$, we denote by $\partial S$ the \emph{boundary} of $S$, i.e. the set of all vertices that are adjacent to $S$ but not themselves in $S$. For our purposes, it is convenient to define $\partial\emptyset=\{\text{root}\}$.

\begin{defn}\label{defn:polA}
 For a rooted tree $T$, denote by $S_T(x,y)$ (resp. $A_T(x,y)$) the bivariate generating function for subtrees (resp. admissible subtrees) of $T$ according to their number of vertices and boundary vertices. That is,
 \begin{equation}\label{eq:Sdef}
  S_T(x,y)=\sum_{T'\ssq T} x^{|T'|} y^{|\bdry T'|}
 \end{equation}
 and 
 \begin{equation}\label{eq:Adef}
  A_T(x,y)=\sum_{T'\in\msc A(T)} x^{|T'|} y^{|\bdry T'|}. 
 \end{equation}
 If $F$ is a rooted forest having components $T_1,\dots,T_r$, then define
 \begin{equation}\label{eq:def_extension}
  S_F(x,y):=\prod_{i=1}^r S_{T_i}(x,y) \qquad \text{ and }\qquad A_F(x,y):=\prod_{i=1}^r A_{T_i}(x,y).
 \end{equation}
\end{defn}

For example, if $T$ is the path on $n$ vertices, again with the root located at one of the endpoints, then any shorter path starting at the root is a non-empty admissible subtree, and thus $A_T(x,y)=y(1+x+\dots+x^{n-1})$. Additionally, the entire path itself is the only non-admissible subtree (with $n$ vertices and empty boundary), so $S_T(x,y)=A_T(x,y)+x^n$. On the other hand, for $T$ being the centrally-rooted star on $n+1$ vertices, we have only two admissible subtrees and obtain $A_T(x,y)=y+xy^n$, but $S_T(x,y)=y+x(x+y)^n$.

\subsection*{The graph at separation}

The \emph{fringe subtree $T_v$} of a rooted tree $T$ is the induced subgraph of $T$ consisting of the vertex $v$ (which is designated the root of $T_v$) and all descendants of $v$. The following definition can be thought of as a weighted version of $A_T$, where each monomial summand stemming from an admissible subtree $T'$ gets a weight depending on the fringe subtrees rooted at $\bdry T'$. The particular choice of the weighing stems from the probabilistic interpretation of this polynomial, which will be elaborated upon in Section~\ref{sec:prob} below, and in particular from equation~\eqref{eq:pgfTS}.

\begin{defn}\label{defn:polM}
 For a rooted tree $T$, denote by $M_T(x,y,z)$ the trivariate polynomial defined by
 \begin{equation}\label{eq:Mdef}
  M_T(x,y,z)=\sum_{T'\in\msc A(T)} x^{|T'|}y^{|\bdry T'|-1} \sum_{v\in\bdry T'}\frac{1}{z}p_v(z),
 \end{equation}
 where $p_v(z):=p_{T_v}(z)=1-P_{T_v}(z,-1)$. 
\end{defn}

It follows from either Lemma~\ref{lemma:properties} below or from the probabilistic interpretation of $p_v$ that $p_v(0)=0$, so $\frac{1}{z}p_v(z)$ is indeed a polynomial in $z$.

\section{The probabilistic viewpoint: Random destruction of trees}\label{sec:prob}

We use this section to explain how the polynomials introduced in Section~\ref{sec:defs} relate to, and are inspired by, probabilistic considerations. 

\subsection*{Random destruction of trees}
Two popular models for randomly destroying graphs are percolation and the cutting model. We use this section to give a very brief introduction to key notions for both of these models, in order to provide a probabilistic motivation for studying the polynomial invariants of this paper in the section below. 

In $\Ber(q)$-site percolation, a probability $q\in[0,1]$ is fixed, and every vertex in a fixed underlying graph is deleted with probability $1-q$ and otherwise kept, independently from all other vertices. The connected components of the induced subgraph of all the vertices that are being kept are called clusters. Bernoulli site percolation can be seen as a continuous-time process in $q\in[0,1]$, by virtue of the following coupling: Equip every vertex $v$ with an independent random variable $X_v$ having the uniform distribution on $[0,1]$. At time $q$, a vertex $v$ is deleted if and only if $X_v>q$, and otherwise kept. It follows immediately that through this coupling, we may assume that $\Ber(q)$-site percolation produces a subgraph of $\Ber(q')$-site percolation whenever $q<q'$. Percolation has been extensively studied, and we refer to \cite{grimmett1999percolation} as a general reference. 

In the cutting model on a rooted tree $T$, vertices are deleted (i.e. cut) randomly one at a time, and all components not containing the root node are immediately discarded. This process necessarily stops once the root node is cut. Equivalently, one can equip each vertex $v$ in $T$ with an independent alarm clock ringing at a uniformly random time $X_v\in[0,1]$, at which the vertex $v$ is cut. It is easy to see that this continuous-time cutting model, as $t$ increases from 0 to 1, is exactly the evolution of the cluster containing the root node in the coupling described above for $\Ber(1-t)$-site percolation. The cutting model has first been considered by Meir and Moon in \cite{MeirMoon1970}, but has received significant attention in the last two decades through works such as \cite{panholzer2006cutting,janson2006random,bertoin2012fires,addario2014cutting}, just to name a few. 

For the cutting model on rooted trees, we say that \emph{separation} occurs at the first time when the remaining tree does not contain any original leaf of $T$ anymore. The remaining tree at this point in time will be denoted by $T_{\mf S}$ (cf. \cite{Bur21modification}). Note that $T_{\mf S}$ does not depend on whether we are working in discrete or continuous time. The admissible subtrees introduced in Section~\ref{sec:defs} are precisely those subtrees $T'$ of $T$ such that $\bP[T_{\mf S}=T']>0$, where $\bP$ denotes the probability measure stemming from the random cutting model on $T$.

\subsection*{Interpretation of the polynomials}
Using the connection described above between percolation and continuous-time cuttings, we note that the probability that $\Ber(q)$-site percolation for $q\in[0,1]$ contains a path from the root to a leaf equals the probability that separation has not occurred by time $1-q$ in the continuous-time cutting model. By virtue of Propositions~6 and 7 in \cite{Bur21modification}, this probability is given by $p_T(q)$ which is a polynomial in $q$ whose coefficients are given as 
\[
 [q^k]p_T(q) = \sum_{\stackrel{T'\ssq T \text{ leaf-induced}}{|T'|=k}} (-1)^{|L(T')|+1}.
\]
The polynomial $P_T$ is then obtained through a bivariate extension, such that the second variable replaces the sign and we obtain a generating function as in Definition~\ref{defn:pol}. 

In the setting of $\Ber(q)$-site percolation on a rooted tree $T$, the term $q^{|T'|}(1-q)^{|\bdry T'|}$ gives the probability that a subgraph $T'$ of $T$ is the root cluster of the percolation. The restriction to admissible subgraphs in \eqref{eq:Adef} leads to connections between $A_T$ and the polynomials $p_T$ and $M_T$, see Lemma~\ref{lemma:MAp}, and is more relevant to the study of the random cutting model. While the change from the $S_T(q,1-q)$ to the bivariate invariant $S_T(x,y)$ (and analogously for $A_T$) might seem like an ad-hoc generalization, it has its motivation in enabling the recursions in Lemma~\ref{lemma:SArec}.

In the case where $S$ and $A$ are applied to rooted forests, defined in \eqref{eq:def_extension}, it is still possible to relate these polynomials to the random destruction of rooted forests in a matter analogous to the case of trees, but we will omit the details here. 

Assume that the continuous-time cutting model separates $T$ at some time $q_0\in[0,1]$, and leaves behind an admissible graph $T'$. Then immediately before separation, all but one of the vertices in $\bdry T'$ must have been cut already, with the exceptional vertex $v\in\bdry T'$ being such that there still is a path connecting the root to a leaf through $v$ present. Moreover, none of the vertices in $T'$ can have been cut before $q_0$. In particular, at time $q_0$, the fringe subtree $T_v$ has not yet been separated itself. Employing this idea, it is possible to show that 
\[
 \bP[T_{\mf S}=T']=\int_0^1 u^{|T'|-1}(1-u)^{|\bdry T'|-1}\sum_{v\in\bdry T'} p_v(u)\,\Di u
\]
(cf. Proposition~5 in \cite{Bur21modification}). From this, it follows immediately that the probability generating function of $|T_{\mf S}|$ is given by 
\begin{equation}\label{eq:pgfTS}
 \sum_{n\geq0} \bP[|T_{\mf S}|=n]x^n = \int_0^1 M_T(xu,1-u,u)\,\Di u
\end{equation}
It might therefore seem more useful to directly investigate the polynomial on the right-hand side of \eqref{eq:pgfTS}; however, a possible advantage of $M_T$ lies in the recursion \eqref{eq:Mrec}.

\section{Some identities}\label{sec:ident}

The purpose of this section is to exhibit recursion formulas for all relevant polynomials, as well as identities relating the polynomials to one another. The following first lemma will prove useful throughout:

\begin{lemma}\label{lemma:properties}
 Let $F$ be a rooted forest. Then:
 \begin{enumerate}[(a)]
  \item The number of vertices of $F$ equals $\deg_x(P_F)$.
  \item The number $\ell$ of leaves of $F$ equals $\deg_y(P_F)$.
  \item Specializing to $x=1$ gives $P_F(1,y)=(1+y)^\ell$. In particular, we have $P_F(1,-1)=0$ unless $F$ is the empty forest, in which case $P_F\equiv 1$. 
 \end{enumerate}
\end{lemma}
\begin{proof}
 Parts (a) and (b) are immediate from Definition \ref{defn:pol}. For part (c), note that $P_F(1,y)$ is the generating function for leaf-induced subforests with a given number of leaves. Since subsets of leaves are in bijection with leaf-induced subforests, we have $P_F(1,y)=\sum_{k=0}^\ell\binom{\ell}{k}y^k=(1+y)^\ell$.
\end{proof}

\begin{lemma}\label{lemma:polrec}
 We have $P_{\bullet}(x,y)=1+xy$ and $p_{\bullet}(x)=x$. Let $F$ be a non-empty rooted forest with rooted trees $T_1,\dots,T_r$ ($r\geq 1$) as its components. 
 \begin{enumerate}[(a)]
  \item We then have 
   \begin{equation}\label{eq:Pforest}
    P_F(x,y)=\prod_{i=1}^r P_{T_i}(x,y).
   \end{equation}
  \item For $T=\wedge(F)$, that is, for a tree having branches $T_1, \dots,T_r$, we have
   \begin{equation}\label{eq:Prec}
    P_T(x,y)=1-x+xP_F(x,y).
   \end{equation}
  \item As a consequence, 
   \begin{equation}\label{eq:prec}
    p_T(x) = x\left(1-\prod_{i=1}^r (1-p_{T_i}(x))\right). 
   \end{equation}
 \end{enumerate}
\end{lemma}

\begin{proof}
 For part (a), let $F'$ be any leaf-induced subforest of $F$. Then the intersections $F'\cap T_1,\dots,F'\cap T_r$ are (possibly empty) leaf-induced subtrees of $T_1,\dots,T_r$, respectively. In this way, we can identify $F'$ with the $r$-tuple $(F'\cap T_1,\dots,F'\cap T_r)$, and both the number of vertices and the number of leaves in these components add up to the respective numbers of $F'$. Thus the bivariate generating function $P_F$ equals the product $\prod_{i=1}^r P_{T_i}$.
 
 For part (b), observe that there is a bijection between non-empty leaf-induced subforests of $F$ and leaf-induced subtrees of $\wedge(F)$, simply by adding the root node of $\wedge(F)$ to the subforest of $F$. Since this increases the number of vertices by 1, the generating function for those subtrees is given by $x\left(P_F(x,y)-1\right)$. Accounting for the empty subforest of $\wedge(F)$ as well yields the result.
 
 Finally, part (c) follows from (a) and (b) after recalling the definition $p_T(x)=1-P_T(x,-1)$. 
\end{proof}

\begin{lemma}\label{lemma:SArec}
 We have $S_{\bullet}(x,y)=y+x$ and $A_{\bullet}(x,y)=y$. If $T$ is a rooted tree with branches $T_1,\dots,T_r$, then 
 \begin{equation}\label{eq:Srec}
  S_T(x,y) = y+x\prod_{i=1}^r S_{T_i}(x,y)
 \end{equation}
 and 
 \begin{equation}\label{eq:Arec}
  A_T(x,y) = y+x\prod_{i=1}^r A_{T_i}(x,y).
 \end{equation}
\end{lemma}

\begin{proof}
 The claims for the tree on one vertex are easily verified from the definitions. 
 
 Consider a subtree $T'$ of $T$. Then $T'$ is either empty, or it consists of the root together with the parts belonging to individual branches, $T_i'=T'\cap T_i$, for $i=1,\dots,r$. In the non-empty case, $T'$ is uniquely determined by the $T_i'$, and we have $|T'|=1+\sum_i|T'_i|$ and $|\bdry T'|=\sum_i|\bdry T'_i|$. Thus,
 \begin{align*}
  S_T(x,y)&=\sum_{T'\ssq T} x^{|T'|}y^{|\bdry T'|} 
  =y + x \!\! \sum_{T'_1\ssq T_1} \!\! \cdots \!\! \sum_{T'_r\ssq T_r} \!\! x^{\sum_\ell |T'_\ell|} y^{\sum_\ell |\bdry T'_\ell|}\\
  &= y + x \prod_{i=1}^r \sum_{T'_i\ssq T_i} x^{|T'_i|}y^{|\bdry T'_i|}
  = y+x\prod_{i=1}^r S_{T_i}(x,y),
 \end{align*}
 which proves \eqref{eq:Srec}.
 
 Note that if $T'$ is an admissible subtree of $T$, then the corresponding $T'_i$ will be admissible subtrees of $T_i$, for each $i$. Conversely, any non-empty $T'$ is again uniquely determined by the $T'_i$. Hence, the computations for equation~\eqref{eq:Arec} are identical to the ones above. 
\end{proof}

\begin{lemma}\label{lemma:Mrec}
 We have $M_{\bullet}(x,y,z)=1$. If $T$ is a rooted tree with branches $T_1,\dots,T_r$ then 
 \begin{equation}\label{eq:Mrec}
  M_T(x,y,z)=\frac{1}{z}p_T(z) +x\sum_{i=1}^r M_{T_i}(x,y,z)\prod_{j\neq i}A_{T_j}(x,y)
 \end{equation}
\end{lemma}

\begin{proof}
 We use the same approach and notation as in the proof of Lemma~\ref{lemma:SArec}. So, any $T'\in\msc A(T)$ is either empty, or contains the root together with parts $T_i'\in\msc A(T_i)$ for each branch $T_1,\dots,T_r$. Thus, we obtain
 \begin{align*}
  M&_T(x,y,z) 
   = \sum_{T'\in \msc A(T)} x^{|T'|} y^{|\bdry T'|-1} \sum_{v\in\bdry T'}\frac{p_v(z)}{z} \\
  &= \frac{p_T(z)}{z} +\!\! \sum_{T'_1\in\msc A(T_1)}\!\! \cdots \!\! \sum_{T'_r\in\msc A(T_r)} \!\! x^{1+\sum_\ell |T'_\ell|} y^{\sum_\ell \bdry |T'_\ell|-1} \! \sum_{v\in\bigcup_i \bdry T'_i} \frac{p_v(z)}{z}\\
  &= \frac{p_T(z)}{z} + \frac{x}{yz} \sum_{i=1}^r \sum_{T'_1\in\msc A(T_1)}\!\! \cdots \!\! \sum_{T'_r\in\msc A(T_r)} \!\! x^{\sum_\ell |T'_\ell|} y^{\sum_\ell \bdry |T'_\ell|} \sum_{v\in\bdry T'_i}p_v(z)\\
  &= \frac{p_T(z)}{z} + \frac{x}{yz} \sum_{i=1}^r \left(\prod_{j\neq i} \sum_{T'_j\in\msc A(T_j)} \!\!\! x^{|T'_j|}y^{|\bdry T'_j|}\right) \!\! \sum_{T'_i\in\msc A(T_i)}\!\!\! x^{|T'_i|}y^{|\bdry T'_i|} \sum_{v\in\bdry T'_i}p_v(z).
 \end{align*}
 By comparing the final expression to Definitions~\ref{defn:polA} and \ref{defn:polM}, we obtain \eqref{eq:Mrec}. 
\end{proof}

\begin{lemma}\label{lemma:MAp}
 For any rooted tree $T$, we have the following three identities:
 \begin{equation}\label{eq:MderivA}
  M_T(x,y,1)=\frac{\partial}{\partial y} A_T(x,y)
 \end{equation}
 \begin{equation}\label{eq:Aandp}
  A_T(x,1-x) = 1-p_T(x)
 \end{equation}
 \begin{equation}\label{eq:Mderivp}
  M_T(x,1-x,x)=\frac{\Di}{\Di x} p_T(x).
 \end{equation}
\end{lemma}
\begin{proof}
 For the proof of \eqref{eq:MderivA}, consider a vertex $v\in V(T)$. Then $p_v(1)=1-P_{T_v}(1,-1)=1$ by Lemma~\ref{lemma:properties}(c), and we thus have 
 \[
  x^{|T'|}y^{|\bdry T'|-1}\sum_{v\in\bdry T'}\left.\frac{p_v(z)}{z}\right|_{z=1}=|\bdry T'|x^{|T'|}y^{|\bdry T'|-1}
 \]
 for any fixed $T'\in\msc A(T)$. Hence 
 \[
  M_T(x,y,1) = \sum_{T'\in\msc A(T)} |\bdry T'|x^{|T'|}y^{|\bdry T'|-1} = \frac{\partial}{\partial y} A_T(x,y),
 \]
 as required.
 
 The identity \eqref{eq:Aandp} follows immediately from comparing the recursions \eqref{eq:Arec} and \eqref{eq:prec}. 
 
 Equality \eqref{eq:Mderivp} is trivially true for $T=\bullet$, and we will now use an inductive argument: Assuming that the identity holds for any trees $T_1,\dots,T_r$, we will show that it is also true for $T=\wedge(T_1,\dots,T_r)$. To do this, consider the recursion \eqref{eq:prec} and take the derivative:
 \begin{align*}
  \frac{\Di p_T(x)}{\Di x} 
  &= 1-\prod_{i=1}^r (1-p_{T_i}(x)) + x\sum_{i=1}^r \frac{\Di p_{T_i}(x)}{\Di x} \prod_{j\neq i} (1-p_{T_j}(x))\\
  &= \frac{p_T(x)}{x} + x\sum_{i=1}^r M_{T_i}(x,1-x,x) \prod_{j\neq i} A_{T_j}(x,1-x)\\
  &= M_T(x,1-x,x).
 \end{align*}
 For the second equality, we used \eqref{eq:prec}, \eqref{eq:Aandp}, and the induction hypothesis; and the final equality follows from \eqref{eq:Mrec}, the recursion for $M$. 
\end{proof}

Observe that by \eqref{eq:MderivA}, $A_T$ is uniquely determined by $M_T$, since $\eqref{eq:Arec}$ implies that $A_T(x,0)=0$. Moreover, $p_T$ is uniquely determined by $A_T$ according to \eqref{eq:Aandp}.

\section{Two complete invariants}\label{sec:leaf-induced}

As an immediate consequence of Definitions \ref{defn:pol},\ref{defn:polA}, and \ref{defn:polM}, we get that two isomorphic rooted trees $T_1\cong T_2$ have the same polynomials. The aim of this section is to show that the converse is true as well for the polynomials $P$ and $S$. Specifically, we will prove the following theorem:

\begin{thm}\label{thm:complete}
 The polynomials $P$ and $S$ as defined in Definition~\ref{defn:pol} are complete invariants for rooted forests. In other words, for rooted forests $F_1, F_2$ we have $P_{F_1}=P_{F_2}$ or $S_{F_1}=S_{F_2}$ if and only if $F_1\cong F_2$.  
\end{thm}

As pointed out above, it only remains to show that either of the two equalities is sufficient for $F_1\cong F_2$, and we devote the rest of the section to this proof. 

A key ingredient for the proof will be that in a unique factorization domain (\emph{UFD}), polynomials can -- by definition -- be factored uniquely into irreducibles; and we will employ the fact that both $\IZ[x,y]$ and $\IC[x,y]$ are UFDs. 

By the \emph{stem} of a rooted tree, we understand the set of vertices constructed in the following iterative way: Start by including the root node of $T$. If the last included vertex has a unique child, include that child as well. Otherwise stop. In other words, the stem consists of all those vertices between the root and the first ``branching'' of the tree (the two endpoints included). For convenience, we declare the stem of a rooted forest on zero or at least two components to be the empty set.  

\begin{lemma}\label{lemma:stem}
 Let $F$ be a rooted forest. Then, the number $s$ of vertices in the stem of $F$ equals $p'_F(1)=-\left.\frac{\partial P_F}{\partial x}\right|_{(1,-1)}$, with the partial derivative being zero if $F$ is not a tree. 
\end{lemma}

\begin{proof}
 The claim is obviously true for the empty rooted forest. In all other cases, we use induction on $s$, beginning with $s=0$ (i.e. $F$ has at least two components).
 
 For $s=0$, denote by $T_1,\dots,T_r$ for $r\geq2$ the components of $F$. Then $P_{T_i}(1,-1)=0$ for all $i=1,\dots,r$ by Lemma~\ref{lemma:properties}(c), so the polynomial 
 \[
  P_F(x,-1)=\prod_{i=1}^r P_{T_i}(x,-1)
 \]
 has an $r$-fold zero at $x=1$. In particular, $\left.\frac{\partial P_F}{\partial x}\right|_{(1,-1)}=0$.
 
 Assume that we have already shown the statement for some $s\geq 0$. Let $F$ be any rooted tree with $s+1$ vertices in its stem. Then $F=\wedge(F')$ where $F'$ is the forest obtained by removing the root of $F$, and $F'$ is a rooted forest with $s$ stem vertices. In the special case where $F'$ is the empty forest, $F$ is the rooted tree on a single vertex, and we can check directly that $-\left.\frac{\partial (1+xy)}{\partial x}\right|_{(1,-1)}=1$. In any other case, we employ Lemma \ref{lemma:polrec}(b) and the induction hypothesis to obtain 
 \[
  \left.\frac{\partial P_F}{\partial x}\right|_{(1,-1)} 
  = -1 + P_{F'}(1,-1) + \left.\frac{\partial P_{F'}}{\partial x}\right|_{(1,-1)}
  = -1 - s,
 \]
 since $P_{F'}(1,-1)=0$.
\end{proof}

\begin{prop}\label{prop:irred}
 Let $F$ be a non-empty rooted forest. Then, $P_F$ is irreducible in $\IC[x,y]$ if and only if $F$ is a tree. 
\end{prop}

\begin{proof}
 If $F$ is not a tree, then it consists of at least 2 components, each containing at least one vertex. Thus by part (a) in Lemma \ref{lemma:polrec}, $P_F$ factors into non-constant polynomials. 
 
 Now assume that $F$ is a tree on $n\geq 1$ vertices with $s\geq 1$ vertices in its stem, having $\ell\geq 1$ leaves. Assume $P_F=fg$ for $f,g\in\IC[x,y]$. Specializing to $x=1$, we obtain $f(1,y)=(1+y)^{k_1}$ and $g(1,y)=(1+y)^{k_2}$ for $k_1,k_2\geq 0$ with $k_1+k_2=\ell$, according to Lemma~\ref{lemma:properties}(c) and because the factors $1+y$ are irreducible. If both $k_1,k_2>0$ then the product rule dictates
 \[
  -s=
  \left.\frac{\partial P_F}{\partial x}\right|_{(1,-1)} = 
  f(1,-1)\cdot \left.\frac{\partial g}{\partial x}\right|_{(1,-1)} + 
  g(1,-1)\cdot \left.\frac{\partial f}{\partial x}\right|_{(1,-1)}
  =0,
 \]
 a contradiction. Hence, without loss of generality $k_1=0, k_2=\ell$, and so $\deg_y(P_F)=\ell=\deg_y(g)$, which implies $\deg_y(f)=0$. In other words, $f$ can be considered as a univariate polynomial in $x$. 
 
 Now write 
 \[
  P_F(x,y)=a_{\ell}(x)y^\ell + \dots + a_1(x)y+a_0(x)
 \]
 for suitable polynomials $a_0,a_1,\dots,a_\ell\in\IC[x]$. If $f(x)$ is a divisor of $P_F$, it must therefore be a common divisor of $a_0,\dots,a_\ell$. However, from Definition \ref{defn:pol} we infer that $a_0(x)=1$. Thus $f(x)$ is a constant. 
\end{proof}

\begin{prop}\label{prop:irred2}
 Let $F$ be a non-empty rooted forest. Then $S_F$ is irreducible in $\IZ[x,y]$ if and only if $F$ is a tree. 
\end{prop}
\begin{proof}
 If $F$ is not a tree, the reducibility of $S_F$ follows from the definition in~\eqref{eq:def_extension}.
 
 To show irreducibility in the case where $F$ is a rooted tree, we use Eisenstein's criterion (cf. \cite[Proposition~A.5.3]{morandi1996field}) on the integral domain $\mf D:=\IZ[y]$. Since $\IZ[x,y]\cong\mf D[x]$, we can consider the prime ideal $\mf p=\langle y\rangle$ in $\mf D$. Writing $S_F$ as 
 \begin{equation}\label{eq:Syx}
  S_F(x,y)=a_0(y) + a_1(y)x + \dots + a_n(y)x^n
 \end{equation}
 with $a_0,a_1,\dots,a_n\in\mf D$, we note that $a_n=1$ since the highest $x$-degree term in $S_F$ stems from the subgraph that is the entire tree, which contains $n=|V(F)|$ vertices, and no boundary vertices. Hence $a_n\notin \mf p$. Moreover, any smaller subtree $T'\ssq F$ omits a vertex in $F$, and therefore has a vertex adjacent to, but not in $T'$ (in the special case where $T'=\emptyset$, this vertex is the root of $F$). Thus, the strict subtrees all contribute monomials divisible by $y$, and hence $a_0,\dots,a_{n-1}\in\mf p$. Finally, for $T'=\emptyset$ we have $\bdry T'=\{\rt\}$, thus $a_0(y)=y\notin \mf p^2$ (and this is only correct if $F$ is a tree). Therefore $S_F$ cannot be factored into non-constant polynomials in $\mf D[x]$ according to Eisenstein's criterion, and since $a_n=1$ it is even irreducible in $\IZ[x,y]$. 
\end{proof}

We now have all the tools assembled to prove Theorem~\ref{thm:complete}.

\begin{proof}[Proof of Theorem~\ref{thm:complete}]
 Assume first $P_{F_1}=P_{F_2}$. Since the polynomial determines the number of vertices and the number of vertices in the stem, those characteristics of $F_1$ and $F_2$ coincide, and we denote them by $n$ and $s$, respectively, as in the proof of Proposition \ref{prop:irred}. 
 
 Suppose the claim is false. Then there exist non-isomorphic $F_1,F_2$ with $P_{F_1}=P_{F_2}$, and we can consider such a pair with $n$ minimal. If $s\geq1$, then $F_i$ is a tree with root $\rho_i$ (for $i=1,2$), and we can consider $F_1-\rho_1$ and $F_2-\rho_2$ instead. As noted in the previous section, we have $F_i\cong \wedge(F_i-\rho_i)$ for $i=1,2$, so by Lemma~\ref{lemma:polrec}(b) we obtain $P_{F_1-\rho_1}= P_{F_2-\rho_2}$. By the minimality of $F_1,F_2$, it follows that $F_1-\rho_1\cong F_2-\rho_2$, and hence $F_1\cong F_2$, a contradiction. Therefore, the minimal counterexamples $F_1,F_2$ have to be either empty (which is trivially not a counterexample) or forests with at least 2 components each. 
 
 So, denote by $T_1,\dots,T_r$ and $T'_1,\dots,T'_{r'}$ the components of $F_1$ and $F_2$, respectively. Lemma~\ref{lemma:polrec} yields
 \[
  \prod_{i=1}^r P_{T_i} = P_{F_1} = P_{F_2} = \prod_{j=1}^{r'} P_{T'_j}.
 \]
 As we are working in the UFD $\IC[x,y]$ and the factors $P_{T_i}$ and $P_{T'_j}$ are monic irreducibles by Proposition~\ref{prop:irred}, it follows that $r=r'$ and that there is a permutation $\pi\in S_r$ with $P_{T_i}=P_{T'_{\pi(i)}}$ for $i=1,\dots,r$. Invoking again the minimality of $F_1, F_2$, we conclude $T_i\cong T'_{\pi(i)}$, and these isomorphisms can be glued together to an isomorphism $F_1\cong F_2$, which is the desired contradiction.\\
 
 Assume now $S_{F_1}=S_{F_2}$ instead. Observe that $S$ again determines the number $n$ of vertices, and the number $s$ of vertices in the stem. Indeed, $n$ is given as the $x$-degree, and $s=\max(j,n)$, where $j$ is the lowest index such that $\deg_y a_j(y)>1$ when we represent $S$ as in equation \eqref{eq:Syx} (this is because the last vertex in the stem is the closest vertex to the root that has more than one descendant, so the subtree induced by the stem is the smallest subtree to have a boundary with more than one vertex, unless $s=n$). Observe moreover that for a rooted tree $T$, we have $S_{T-\rt}=\frac{1}{x}(S_T-y)$, which follows from comparing the recursion \eqref{eq:Srec} with \eqref{eq:def_extension}. 
 
 With these observations in place, the rest of the argument works entirely analogously to the previous case, except that we work in the UFD $\IZ[x,y]$ (rather than $\IC[x,y]$), due to Proposition~\ref{prop:irred2}.  
\end{proof}

\subsection*{An application to the reconstruction of rooted trees}

The reconstruction conjecture, going back to Ulam \cite{Ulam60} and Kelly \cite{Kelly57}, asks whether every simple graph $G=(V,E)$ on at least 3 vertices is uniquely (i.e. up to isomorphism) determined by the multiset, called \emph{deck}, of its vertex-deleted subgraphs $G-v$ for $v\in V$. It has been widely investigated since these initial papers. In the case of trees, it was already shown in \cite{Kelly57} that they are reconstructible, with stronger results (using fewer subgraphs) obtained in \cite{HP66} and \cite{Bondy69}. Moreover, Ne\v{s}et\v{r}il \cite{Nesetril71} considered a version of tree reconstruction where the deck was instead of the collection of asymmetric maximal proper subtrees. In the same line, we will show in this section that Theorem~\ref{thm:complete} implies that rooted trees are uniquely determined by their inclusion-maximal leaf-induced proper subtrees:

\begin{prop}
 Let $F$ be a rooted forest with $\ell\geq 3$ leaves. Then $F$ can be uniquely reconstructed from its deck $\msc D(F)$ of maximal leaf-induced proper subforests. 
\end{prop}
\begin{proof}
 We will show that we can reconstruct $P_F$ from $\msc D(F)$, the claim then follows from Theorem~\ref{thm:complete}. The maximal leaf-induced proper subforests each contain $\ell-1$ leaves, hence the number $\ell$ is reconstructible from the deck. Observe that a leaf-induced subtree with $k$ leaves is contained in $\ell-k$ trees in $\msc D(F)$, and that thus by Definition~\ref{defn:pol}, we have 
 \[
  [y^k]P_F(x,y) = \frac{1}{\ell-k} \sum_{F'\in \msc D(F)} [y^k]P_{F'}(x,y)
 \]
 for all $0\leq k\leq \ell-1$. Note that the right-hand side is computable given $\msc D(F)$, and hence the same holds true for 
 \[
  \tilde P_F(x,y) := P_F(x,y) - x^{|F|}y^\ell = \sum_{k=0}^{\ell-1} y^k \big( [y^k] P_F(x,y)\big). 
 \] 
 Denote by $s_F$ and $s_{F'}$ the number of vertices in the stem of $F$ and $F'\in \msc D(F)$, respectively. Since we assume $\ell\geq 3$, there exists an $F'$ such that $s_{F'}=s_F$, and hence $s_F=\min_{F'\in \msc D(F)} s_{F'}$. (This is no longer true for $\ell=2$: The graphs in $\msc D(F)$ would then be two paths, each connecting a root to a leaf, and there is no way for us to determine how large the intersection of the two paths in $F$ is.) Accordingly, $s_F$ is reconstructible from $\msc D(F)$, and using Lemma~\ref{lemma:stem} we obtain
 \[
  |F|=(-1)^\ell \left.\frac{\partial x^{|F|} y^\ell}{\partial x}\right|_{(1,-1)} = (-1)^{\ell + 1}\left(s_F + \left.\frac{\partial \tilde P_F(x,y)}{\partial x}\right|_{(1,-1)}\right).
 \]
 The right-hand side is again reconstructible, which implies that $P_F$ is reconstructible, concluding the proof. 
\end{proof}

\begin{rem}
 The author is unaware of a proof that rooted trees are reconstructible from their deck of $\ell$ maximal rooted proper subtrees which -- analogously to the previous Proposition -- makes use of the completeness of $S$. Indeed, given a rooted subtree of some rooted tree $T$, it is not clear which of the leaves are also leaves of $T$, and thus reconstructing $S$ directly from the deck seems difficult. 
\end{rem}

\section{Remarks, examples, and open problems}\label{sec:discussion}

We begin by making a number of remarks, combined with examples and non-examples, concerning the results of Sections~\ref{sec:ident} and \ref{sec:leaf-induced}. 

\begin{figure}
\includegraphics{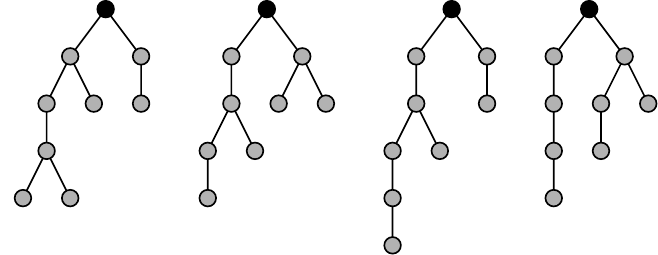}
\caption{Non-isomorphic rooted trees $T_1,T_2,T_3,T_4$ (from left to right).} \label{fig:T14}
\end{figure}

\begin{rem}
 Unlike $P$, the univariate polynomial $p$ is not a complete invariant for rooted trees: As S. Wagner pointed out (\!\!\cite{WagnerCom}), the trees $T_1,T_2$ and $T_3,T_4$ in Figure~\ref{fig:T14} form two pairs of non-isomorphic trees that share the same polynomial, namely
 \begin{align*}
  p_{T_1}(x)=p_{T_2}(x) &= 2x^3+x^5-3x^6-x^7+3x^8-x^9, \ \text{and}\\
  p_{T_3}(x)=p_{T_4}(x) &= x^3+x^4-x^7-x^8+x^9.
 \end{align*}
 In fact, it can be verified by a computer search that these are the smallest such pairs. To exemplify Theorem~\ref{thm:complete}, the corresponding bivariate polynomials are given by
 \begin{align*}
  P_{T_1}&= 1 + 2x^3y + 2x^5y + x^5y^2 + 3x^6y^2 + 2x^7y^2 + x^7y^3 + 3x^8y^3 + x^9y^4\\
  P_{T_2}&= 1 + 2x^3y + x^4y + x^4y^2 + x^5y + 3x^6y^2 + 2x^7y^2 + x^7y^3 + 3x^8y^3 + x^9y^4\\
  P_{T_3}&= 1 + x^3y + x^4y + x^6y + x^6y^2 + x^7y^2 + x^8y^2 + x^9y^3\\
  P_{T_4}&= 1 + x^3y + x^4y + x^5y + x^5y^2 + x^7y^2 + x^8y^2 + x^9y^3,
 \end{align*}
 which are pairwise different. 
\end{rem}

\begin{rem}
 Lemma~\ref{lemma:MAp} implies that $M$ is a stronger invariant (in the sense that it distinguishes more trees) than $A$, and $A$ is a stronger invariant than $p$. In fact, these relations are strict: The trees $T_3$ and $T_4$ from Figure~\ref{fig:T14} are distinguished by $A$ but not by $p$, and the trees $T_1$ and $T_2$ are distinguished by $M$ but not by $A$. Indeed, we have 
 \begin{align*}
  A_{T_1}&=A_{T_2}= y + xy^2 + x^2y^2 + x^2y^3 + 2x^3y^3 + x^4y^3 + x^4y^4 + x^5y^4\\
  A_{T_3}&= y +\! xy^2 +\! 2x^2y^2 +\! x^3y^2 +\! x^3y^3\! +\! 2x^4y^3\! +\! 2x^5y^3\! +\! x^6y^3\\
  A_{T_4}&= y +\! xy^2 +\! x^2y^2 +\! x^2y^3 +\! x^3y^2 +\! 2x^3y^3 +\! x^4y^2 +\! 2x^4y^3 +\! 2x^5y^3 +\! x^6y^3
 \end{align*}
 and 
 \begin{align*}
  M_{T_1} &= \frac{p_{T_1}(z)}{z} + xy(2z+2z^3-3z^4+z^5) + x^2y(1+z+2z^3-3z^4+z^5) \\
  &\quad + x^2y^2(1+z+2z^2-z^3) + 2x^3y^2(3+3z+z^2-z^3)\\
  &\quad + x^4y^2(2+2z-z^2) + x^4y^3(3+z) + 4x^5y^3\\
  M_{T_2} &= \frac{p_{T_2}(z)}{z} + xy(2z+z^3-z^4) + x^2y(3z-z^3) + x^2y^2(2+z^2+z^3-z^4)\\
  &\quad + 2x^3y^2(3+4z-z^3) + x^4y^2(2+2z-z^2) + x^4y^3(3+z) + 4x^5y^3.
 \end{align*}
 In light of these examples, it is worth noting that it is possible to fully describe all trees with 3 leaves that share the same $p_T$ with a different tree. In fact, they are of the structure depicted in Figure~\ref{fig:Tp3l} (but we omit the proof in the interest of brevity). It is then easy to see that these trees will always be distinguished by $A$, since $T$ has an admissible subgraph with $s+3k-\gb-2$ vertices, and 2 boundary vertices; whereas the largest admissible subgraph in $\tilde T$ with 2 boundary vertices contains only $s+3k-2\gb-2$ vertices, hence 
 $\deg_x [y^2]A_T > \deg_x [y^2]A_{\tilde T}$.
 
 In full generality, it appears to be a difficult problem to give a graph-theoretic description for the rooted trees $T$ that have a ``cousin''  $T'$ such that $p_T=p_{T'}$ (or $A_T=A_{T'}$). 
\end{rem}

\begin{figure}
\includegraphics{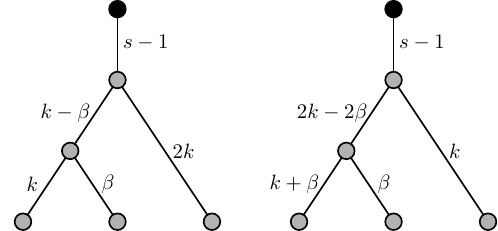}
\caption{The structure of non-isomorphic rooted trees $T$ (left) and $\tilde T$ (right) with 3 leaves with $p_T(x)=p_{\tilde T}(x)=x^s(x^k+x^{2k-\gb}-x^{3k}-x^{4k-\gb}+x^{4k})$. An edge labelled by $w$ indicates a path on $w$ edges. Here, $s$ denotes the number of vertices in the stem, $k\geq2$, and $\gb\in\{1,\dots,k-1\}$.} \label{fig:Tp3l}
\end{figure}

\begin{rem}
 It is worth emphasizing that despite satisfying the same recursion formula -- compare \eqref{eq:Srec} and \eqref{eq:Arec} -- and only differing in their initial values, the polynomial $S$ is a complete invariant, whereas the polynomial $A$ is not. In particular, it follows from the proof of Theorem~\ref{thm:complete} that $A_T$ is reducible for some trees $T$. This is obvious at first glance, since $y$ is a divisor of $A_T$ for every $T$, but this cannot be the only obstacle since otherwise $\frac{1}{y}A$ could be a complete invariant, and therefore also $A$. Indeed, the branches of the trees $T_1$ and $T_2$ from the previous remarks have a more interesting factorization, namely 
 \[
  y(1+xy)(1+x^2y) \quad \text{ and }\quad y(1+x)
 \]
 for the two branches of $T_1$, and 
 \[
  y(1+x)(1+x^2y) \quad \text{ and }\quad y(1+xy)
 \]
 for the two branches of $T_2$. 
\end{rem}

\begin{rem}
 Theorem 3.2 in \cite{Liu21tree} gives a method to obtain a complete invariant for unrooted trees from a complete polynomial invariant for rooted trees that is irreducible in a suitable polynomial ring. The idea is to replace the unrooted tree by a rooted forest that determines the tree up to isomorphism, and then assign to the forest the product of the polynomials of its connected components. While the same idea works for the polynomials of Theorem~\ref{thm:complete}, we prefer to formulate the statement in terms of complete invariants for rooted trees instead. 
\end{rem}

\begin{rem}
 As an anonymous reviewer pointed out, many other polynomial invariants for rooted trees are defined by considering characteristics of either arbitrary edge sets (as in \cite{GM89,CG91Tutte}) or for special classes of subtrees (as in \cite{Liu21tree,RW22polynomial}). The invariant $S$ is special in the sense that it encodes characteristics (the number of vertices and boundary vertices) for all rooted subtrees. This raises the following open question: For which pairs of non-negative, integer characteristics $\ga(T'),\gb(T')$, defined for all subtrees $T'$ of a rooted tree $T$, is the invariant $F_T(x,y)=\sum_{T'\ssq T} x^{\ga(T')} y^{\gb(T')}$ complete for rooted trees? In a similar vein, one might also ask for which kinds of subtrees the polynomial $\sum_{T'} x^{|T'|} y^{|L(T')|}$ is complete. 
\end{rem}

We also state the following conjecture:

\begin{conj}\label{conj}
 The polynomial $M$ defines a complete invariant for rooted trees.
\end{conj}
This has been verified using \texttt{Mathematica} for all rooted trees up to 20 vertices, by evaluating $M$ with Lemma~\ref{lemma:Mrec} for all the trees that are not already distinguished by $A$. However, we at present do not have a proof or counterexample for this conjecture. Moreover, since the recursion formula \eqref{eq:Mrec} for $M$ does not involve a product of the $M_{T_i}$ it seems likely that any proof of the conjecture would require an approach different from the one via irreducibility of polynomials used in the proof of Theorem~\ref{thm:complete}. On a related note, we also do not know if the probability generating function obtained from $M$ in \eqref{eq:pgfTS} is a complete invariant in $\IQ[x]$. Using \texttt{Mathematica} and employing similar considerations as above, this has been checked for all rooted trees on up to 15 vertices. 

In this context, it should be pointed out that each of the polynomials we considered in this paper are either complete invariants of rooted trees; or asymptotically almost all trees on $n$ vertices have a cousin with the same associated polynomial. Indeed, assume one of the invariants $p,P,S,A,M$ is not complete, then there exist rooted trees $T'\ncong T''$ such that both $T'$ and $T''$ are assigned the same value of the invariant. If $T$ is any tree that has a copy of $T'$ as fringe subtree, one can replace that copy by a copy of $T''$ instead. This produces a tree that is indistinguishable from $T$ via the invariant, according to the recursive formulas in Lemmas~\ref{lemma:polrec},\ref{lemma:SArec}, and \ref{lemma:Mrec}. But since asymptotically almost all rooted trees contain a given tree $T'$ as a fringe subtree (this follows e.g. from Theorem 3.1 in \cite{Wag15}, where the additive functional is the number of fringe subtrees isomorphic to $T'$), the proportion of rooted trees with such a cousin will tend to 1. 

In particular, either Conjecture~\ref{conj} holds true, or 
\[
 \bP\left[\left\{\text{There are rooted trees $T'\ncong T''$ on $n$ vertices s.t. }M_{T'}=M_{T''}\right\}\right] \to 1
\]
as $n\to\infty$ where $\bP$ is the uniform probability measure on the set of non-isomorphic rooted trees on $n$ vertices.

\section*{Acknowledgements} 

The author wishes to thank his academic advisors at Uppsala University, Cecilia Holmgren and Svante Janson, for their generous support throughout. The gratitude is extended to Stephan Wagner for the many helpful comments and discussions, and to the anonymous referees for their helpful comments. This work was partially supported by grants from the Knut and Alice Wallenberg Foundation, the Ragnar S\"oderberg Foundation, and the Swedish Research Council. The author has received funding from the European Union's Horizon 2020 research and innovation programme under the Marie Sk\l{}odowska-Curie grant agreement Grant Agreement No 101034253.


\bibliographystyle{alphaurl}
\bibliography{pol2}

\end{document}